\documentclass[12pt, a4paper]{amsart}
\usepackage{amscd,amsmath,amssymb,amsthm,amsfonts}
\usepackage[alphabetic]{amsrefs}
\usepackage{mathrsfs}
\usepackage[shortlabels]{enumitem}
\usepackage[all]{xy}

\usepackage{xcolor}
\usepackage[hypertexnames=true, citecolor = green, colorlinks = false, linkcolor = red, linktoc = none, pdffitwindow = false, urlbordercolor = white]{hyperref}%

\def\ker{\operatorname{ker}}

\def\max{\operatorname{max}}

\def\N{\mathbb{N}}

 \newcommand{\IF}[0]{\mathbb{F}}

\newcommand{\CQ}[0]{\mathcal{Q}}

\newtheorem{thm}{Theorem}[section]
\newtheorem{corollary}[thm]{Corollary}
\newtheorem{lemma}[thm]{Lemma}

\newtheorem{proposition}[thm]{Proposition}

\theoremstyle{definition}
\newtheorem{definition}[thm]{Definition}

\theoremstyle{remark}
\newtheorem{remark}[thm]{Remark}

\numberwithin{equation}{section}

\begin{document}

\title[Graph products and the absence of property (AR)]{Graph products and the absence of property (AR)}

\author{Nicolai Stammeier}
\address{Department of Mathematics \\ University of Oslo \\ P.O.~Box 1053 Blindern \\ NO-0316 Oslo, Norway}
\email{nicolsta@math.uio.no}
\thanks{The author was supported by RCN through FRIPRO 240362.}
\subjclass[2010]{20M10 (Primary) 20F36, 46L05 (Secondary)}

\begin{abstract}
We discuss the internal structure of graph products of right LCM semigroups and prove that there is an abundance of examples without property (AR). Thereby we provide the first examples of right LCM semigroups lacking this seemingly common feature. The results are particularly sharp for right-angled Artin monoids. 
\end{abstract}
\maketitle

\section{Introduction}\label{sec:intro}
The starting point of a number of recent breakthroughs in the theory of semigroup $C^*$-algebras is the seminal work \cites{Li1,Li2}, in which a universal $C^*$-algebra $C^*(S)$ is associated to every left cancellative monoid $S$. In the last years, a particular line of research focused on left cancellative monoids for which the intersection of two principal right ideals is either empty, or another principal right ideal again. Such monoids are called \emph{right LCM semigroups}, and they form an intriguing and tractable class of examples in between positive cones in quasi-lattice ordered groups and general left cancellative monoids, see \cite{BLS1}*{Lemma~3.3 and Corollary~3.6} for details.

Inspired by the treatment of the quasi-lattice ordered case in \cite{CrispLaca}, a \emph{boundary quotient} $\CQ(S)$ of $C^*(S)$ was introduced for right LCM semigroups $S$ in \cite{BRRW}. Soon thereafter, Starling provided an in-depth analysis of $\CQ(S)$ in \cite{Star}, relying on major advances in the understanding of the connections between inverse semigroups, groupoids, and $C^*$-algebras stemming from \cites{EP1,EP2,Ste}. In \cite{BS1}, it was shown that the boundary quotient has a more accessible presentation if the right LCM semigroup has the so-called \emph{accurate refinement property}, henceforth abbreviated \emph{property (AR)}. This property is an analogue of $0$-dimensionality for topological spaces in the context of semigroups, and is enjoyed by various examples, see \cite{BS1}*{Section~2 and Corollary~3.11}. 

The presence of property (AR) was found to be useful in the construction of a \emph{boundary quotient diagram} for right LCM semigroups in the spirit of \cite{BaHLR}, see \cite{Sta3}. This diagram sets the grounds for a unifying approach to the study of equilibrium states on $C^*$-algebras in \cite{ABLS}, where remarkable results on the structure of KMS-states on $C^*(S)$ were obtained for right LCM semigroups satisfying an admissibility condition which implies property (AR), see Subsection~\ref{subsec:right LCM}. Working with abstract right LCM semigroups as opposed to explicit classes of examples allowed for a unification of the inspiring case studies \cites{LR2,BaHLR,LRR,LRRW,CaHR}, and also for coverage of a substantial amount of new examples, most notably, algebraic dynamical systems. Moreover, the techniques in \cites{BS1,Sta3,ABLS} raise several questions on the structure of right LCM semigroups, perhaps most notably:

\begin{enumerate}[(a)]
\item Are there right LCM semigroups without property (AR)?
\item Which right LCM semigroups are admissible? 
\end{enumerate}

The aim of the present work is to investigate in how far graph products of right LCM semigroups as considered in \cites{Cos,FK} provide answers to these two questions. In addition, we also address structural aspects related to the distinguished subsemigroups $S^*, S_c$ and $S_{ci}$. We apply our results to the classical case of right-angled Artin monoids $A_\Gamma^+$ given by an undirected graph $\Gamma$ since many graph related phenomena can already be witnessed here. Indeed, the explicit presentation of the boundary quotient in \cite{CrispLaca}*{Corollary~8.5} involving only the vertex sets of the finite coconnected components of the graph $\Gamma$ may be regarded as an indication for a particularly accessible structure of foundation sets. Another motivation comes from the elegant solution to the isomorphism problem for $C^*(A_\Gamma^+)$, see \cite{ELR}.

Since property (AR) is known for various kinds of right LCM semigroups, we were struck by surprise to find that a right-angled Artin monoid $A_\Gamma^+$ has property (AR) if and only if all of its finitely generated direct summands are free, see Corollary~\ref{cor:AR for ra Artin monoids}. In terms of the $\Gamma$, this means that all finite  coconnected components $\Gamma_i$ do not contain any edges. The result follows from more general graph product considerations in Corollary~\ref{cor:AR for graph products} that rely on Theorem~\ref{thm:lack of acc FS}, where we show that graph products over infinite coconnected graphs have no foundation set other than the obvious ones containing an invertible element, while the analogous statement holds in the finite case for accurate foundation sets.

The characterisation of property (AR) for right-angled Artin monoids $A_\Gamma^+$ in Corollary~\ref{cor:AR for ra Artin monoids} allows us to determine when $A_\Gamma^+$ is admissible in the sense of \cite{ABLS}. It turns out that admissibility and the existence of a generalised scale coincide for right-angled Artin monoids, see Corollary~\ref{cor:gen scale for ra Artin monoids} and Corollary~\ref{cor:admissible ra Artin monoids}. If existent, the generalised scale on $A_\Gamma^+$ is unique and arises as the product of the unique generalised scales on its non-abelian direct summands $A_{\Gamma_i}^+$, see Proposition~\ref{prop:gen scale on direct sum of free monoids} and Corollary~\ref{cor:gen scale for ra Artin monoids}. 

Thus we are lead to the conclusion that graph products of right LCM semigroups mostly lack property (AR), and are therefore not admissible in the sense of \cite{ABLS}. While this rules out the possibility of applying \cite{ABLS} to graph products of right LCM semigroups in great generality, we obtain a fairly detailed description of the behaviour of graph products with respect to the subsemigroups $S_c$ and $S_{ci}$, see Theorem~\ref{thm:reg prop for graph products}. These result show that the graph product represents a useful tool to construct new, and potentially very interesting examples of right LCM semigroups that are well-behaved to some degree, but demand more sophisticated techniques then those applicable to right LCM semigroups that have property (AR) or even a generalised scale. That is why we feel that this work might stimulate further research in the direction of inverse semigroups and groupoids related to (right LCM) semigroups and their $C^*$-algebras.

\emph{Acknowledgements:} The author thanks Nathan Brownlowe, Nadia Larsen, and Adam S{\o}rensen for helpful conversations.

\section{Preliminaries}\label{sec:background}
Here we provide the prerequisites we shall need concerning right LCM semigroups and graph products.
 
\subsection{Right LCM semigroups}\label{subsec:right LCM}
A left cancellative semigroup $S$ is called \emph{right LCM} if the intersection of two principal right ideals in $S$ is either empty or a principal right ideal. For $s,t \in S$, we say that $s$ and $t$ are \emph{orthogonal} and write $s\perp t$ if $sS \cap tS = \emptyset$. Unless specified otherwise, we will always assume that a right LCM semigroup $S$ has an identity, i.e.~$S$ is a monoid. 

Let us first discuss property (AR). A finite subset $F \subset S$ is called a \emph{foundation set} for $S$ if for every $s \in S$ there is $f \in F$ such that $f \not\perp s$, see \cite{BRRW}*{Section~5}. A subset $F\subset S$ is \emph{accurate} if $f\perp f'$ for all $f,f' \in F, f\neq f'$, see \cite{BS1}*{Definition~2.1}. If $F,F'$ are foundation sets such that $F' \subset FS$, then $F'$ is called a \emph{refinement} of $F$. We then say that $S$ has the \emph{accurate refinement property}, or property (AR), if every foundation set for $S$ has an accurate refinement, see \cite{BS1}*{Definition~2.3}.

For a right LCM semigroup $S$, its subgroup of invertible elements shall be denoted by $S^*$. This subgroup lies inside the \emph{core subsemigroup} $S_c:=\{ a \in S \mid a \not\perp s \text{ for all } s \in S\}$, which was first considered for right LCM semigroups in \cite{Star}, but stems from \cite{CrispLaca}*{Definition~5.4}. We remark that $S_c$ is again a right LCM semigroup. Furthermore, it induces an equivalence relation $s \sim t :\Leftrightarrow sa=tb$ for some $a,b \in S_c$ called the \emph{core relation}. In contrast to $S_c$, we also consider the subsemigroup $S_{ci}$ of \emph{core irreducible} elements, that is, the collection of all elements $s \in S\setminus S_c$ for which every factorization $s=ta$ with $t \in S, a \in S_c$ satisfies $a \in S^*$. While $S_{ci}$ does not have an identity by construction, its unitisation $S_{ci}^1 := S_{ci}\cup\{1\}$ and $S_{ci}' := S_{ci} \cup S^*$ do. 

A right LCM semigroup $S$ is called \emph{core factorable} if $S=S_{ci}^1S_c$. We say that $S_{ci} \subset S$ is \emph{$\cap$-closed} if $sS \cap tS = rS$ implies $r \in S_{ci}$ whenever $s,t \in S_{ci}$. To provide some indication why this property is of interest, let us mention that $S_{ci} \subset S$ is $\cap$-closed if and only if $S_{ci}'$ is right LCM and its inclusion into $S$ is a homomorphism of right LCM semigroups, i.e.~it preserves intersections of principal right ideals, see \cite{ABLS}*{Proposition~3.3}. Finally, a nontrivial homomorphism $N\colon S \to \N^\times$ is called a \emph{generalised scale} if $|N^{-1}(n)/_\sim| = n$ and every minimal complete set of representatives for $N^{-1}(n)/_\sim$ forms an accurate foundation set for $S$ for all $n \in N(S)$. Every generalised scale $N$ satisfies $\ker N = S_c$ by \cite{ABLS}*{Proposition~3.6(i)}, and the existence of a generalised scale entails vital information on the structure of $S$. For instance, it implies that the right LCM semigroup has property (AR), see  \cite{ABLS}*{Proposition~3.6(v)}.

Finally, we recall from \cite{ABLS}*{Definition~3.1} that a right LCM semigroup $S$ is called \emph{admissible}, if it is core factorable, $S_{ci} \subset S$ is $\cap$-closed, and $S$ admits a generalised scale $N$ such that $N(S) \subset \N^\times$ is freely generated by its irreducible elements.

\subsection{Graph products}\label{subsec:graph products}
Within this work, a \emph{graph} will mean a countable, undirected graph $\Gamma = (V,E)$ without loops or multiple edges. The concept of a \emph{graph product of groups} emerged in \cite{Gre} as a generalization of \emph{graph groups}, and has been transferred to the setting of monoids in \cite{Cos}: For a graph $\Gamma = (V,E)$ and a family of monoids $(S_v)_{v \in V}$, the \emph{graph product} is the monoid $S_\Gamma$ obtained as the quotient of the direct sum $\bigoplus_{v \in V} S_v$ by the congruence generated by the relation $(st,ts)$ if $s \in S_v, t \in S_w$ with $(v,w) \in E$, see \cite{Cos}*{Section~2} and \cite{FK}*{Section~1}. Given a graph $\Gamma$, its \emph{right-angled Artin monoid} $A_\Gamma^+$ is the graph product with $S_v = \N$ for all $v \in V$. These monoids have also been studied under the names of graph monoids, free partially commutative monoids, and trace monoids, see for instance \cite{Die}. If one switches the vertex monoids from the natural numbers to the integers, the resulting graph product is the \emph{right-angled Artin group} $A_\Gamma$ associated to $\Gamma$, see \cite{Cha} for more.

It was shown in \cite{CrispLacaT} that the graph product is well-behaved with respect to quasi-lattice orders. Invoking a characterization of the right LCM property via the inverse hull semigroup, Fountain and Kambites showed that this can be generalised to right LCM semigroups, see \cite{FK}*{Theorem~2.6}, where we note that we can move back and forth between right cancellative, left LCM semigroups (used in \cite{FK}) and left cancellative, right LCM semigroups by passing to the opposite semigroup.

According to \cite{FK}*{Theorem~1.1}, which is an adaptation of the corresponding result in \cite{Gre}, every element $s$ in a graph product $S_\Gamma$ is represented by an essentially unique \emph{reduced expression} $s_{v(1)}s_{v(2)}\cdots s_{v(n)}$, that is, $s_{v(k)} \in S_{v(k)}, v(k) \neq v(k+1)$, and whenever there are $1 \leq k<m \leq n$ such that $v(k)=v(m)$, then there exists $k < \ell < m$ such that $(v(k),v(\ell)) \notin E$. The analogous result had been proven in the quasi-lattice ordered case before, see \cite{CrispLacaT}*{Theorem~2}. The reduced expression is unique in the sense that any two reduced expressions for the same element are \emph{shuffle equivalent}, i.e.~we can move from one to the other by a finte number of switches of neighbouring factors whose vertices are adjacent in $\Gamma$. Thus there exists a subadditive function $\ell\colon S_\Gamma \to \N$ that assigns the length of any reduced expression to the element in question.

A graph $\Gamma$ is said to be coconnected if there exists no partition $V = V_1 \sqcup V_2$ with $V_i \neq \emptyset$ and $V_1\times V_2 \subset E$. Equivalently, $\Gamma$ is coconnected if the opposite graph $\Gamma^{\text{opp}} := (V, V\times V\setminus (E\cup \{(v,v) \mid v \in V\}))$ is connected. The decomposition of $\Gamma$ into its \emph{coconnected} components is the initial step in the analysis of $S_\Gamma$, see for instance \cite{ELR}, where the synonym \emph{co-irreducible} is used. Every graph $\Gamma$ has a unique decomposition into coconnected components, which we denote by $(\Gamma_i)_{i \in I}$ with $\Gamma_i = (V_i,E_i)$. The original graph can be recovered from $(\Gamma_i)_{i \in I}$ as $V = \bigsqcup_{i \in I} V_i$ and $E = \{ (v,w) \in V\times V \mid (v,w) \in E_i \text{ or } w \notin V_i \ni v \text{ for some } i \in I\}$. It follows from this observation that $S_\Gamma$ coincides with the direct sum $\bigoplus_{i \in I} S_{\Gamma_i}$ over the graph products obtained from its coconnected components. 

A vertex $v \in V$ is called \emph{isolated} if $v$ does not emit any edge, and \emph{universal} if $v$ is connected to every other vertex in $\Gamma$. We note that the only coconnected graph with a universal vertex $v$ is $V=\{v\}$, and that any graph containing an isolated edge is necessarily coconnected. For convenience, we let $V_u$ denote the set of universal vertices, and $I_2 := \{ i \in I \mid \lvert V_i\rvert \geq 2 \}$.

We will make use of the following notion of a blocking path, that is actually a path in the opposite graph.

\begin{definition}\label{def:blocking path}
Let $\Gamma=(V,E)$ be a graph and $C \subset V$. A \emph{blocking path} for $C$ is a finite sequence of vertices $w(1),\ldots,w(n) \in V$ such that 
\begin{enumerate}[(a)]
\item $w(1) \notin C$, $(w(k),w(k+1)) \notin E$ for all $1\leq k \leq n-1$, and
\item for every $u \in C$ there exists $1\leq k\leq n$ such that $(w(k),u) \notin E$.
\end{enumerate}
\end{definition}

It turns out that blocking paths are almost always available whenever the graph is coconnected, and we will frequently make use of this elementary observation in the course of this work.

\begin{lemma}\label{lem:blocking paths in coconnected graphs}
If $\Gamma$ is a coconnected graph with at least two vertices, then every finite proper subset $C$ of $V$ admits a blocking path ending in any prescribed vertex.
\end{lemma}
\begin{proof}
Let $C=\{v(1),\ldots,v(m)\}\subset V$ be finite and proper, that is, $V\setminus C \neq \emptyset$. If $(v,u) \in E$ for all $v \in C, u \in V\setminus C$, then we would get a contradiction to $\Gamma$ being coconnected. Thus there exists $w(1) \in V\setminus C$ such that $(v(k),w(1)) \notin E$ for some $1 \leq k\leq m$. Without loss of generality, we can assume $k=1$. Since $\Gamma$ is coconnected, we can choose $w'(k) \in V$ for $2 \leq k\leq m$ such that $(v(k),w'(k)) \notin E$. Again by coconnectedness, there exists a finite path $w(1),\ldots,w(n)$ in $\Gamma^{\text{opp}}$ that visits every $w'(k), 2 \leq k \leq n$. This is a blocking path for $C$, and since $\Gamma^{\text{opp}}$ is connected, we can attach to this blocking path a path leading to any prescribed vertex without loosing the blocking property for $C$.
\end{proof}

\begin{remark}\label{rem:blocking paths}
Let $\Gamma=(V,E)$ be a graph and $(S_v)_{v \in V}$ a family of right LCM semigroups. Suppose $w(1),\ldots,w(n)$ is a blocking path for some nonempty $C$, and we can choose $s_n,t_n \in S_{w(n)}^{\phantom{*}}\setminus S_{w(n)}^*$. Then for all $s_0,t_0$ whose reduced expressions only contain parts from vertex semigroups of vertices in $C$, and all $s_k,t_k \in S_{w(k)}, 1 \leq k <n$, we have $\ell(s_0s_1\cdots s_n) = \ell(s_0)+n$ and $s_0s_1\cdots s_n \perp t_0t_1\cdots t_n, \text{ unless } s_k=t_k \text{ for } 0 \leq k<n \text{ and } s_n\not\perp t_n$. Thus blocking paths allow for the construction of shuffle inert elements in graph products, which turns out to be quite useful.
\end{remark}

\section{The internal structure of graph products}\label{sec:structure}
In this section we show that many of the properties of $S_\Gamma$ that are of interest to us, e.g.~in connection with \cite{ABLS}, can be understood from a study of the corresponding graph products for the coconnected components $(\Gamma_i)_{i \in I}$ of $\Gamma$. The reason is $S_\Gamma = \bigoplus_{i \in I} S_{\Gamma_i}$ and the following list of straightforward observations, where we write $s=\oplus_{i \in I} s_i$ for $s \in \bigoplus_{i \in I} S_i$:

\begin{proposition}\label{prop:reg properties for direct sums}
Let $(S_i)_{i \in I}$ be a family of right LCM semigroups. Then $S:= \bigoplus_{i \in I}S_i$ has the following features:
\begin{enumerate}[(i)]
\item $S^* = \bigoplus_{i \in I} S_i^*$, $S_c = \bigoplus_{i \in I} (S_i)_c$, and $S_{ci}' = \bigoplus_{i \in I} (S_i)_{ci}'$.
\item $s,t \in S$ are core related if and only if $s_i$ and $t_i$ are core related in $S_i$ for all $i \in I$. 
\item The following statements hold for $S$ if and only if their analogues hold for all $S_i$: $S$ is core factorable, $S_{ci}\subset S$ is $\cap$-closed, $\alpha\colon S_c \curvearrowright S/_\sim, a.[s]:= [as]$ is faithful, and $S$ has finite propagation.
\item The action $\alpha\colon S_c \curvearrowright S/_\sim, a.[s]:= [as]$ is almost free if and only if one of the following conditions holds:
	\begin{enumerate}[(a)]
	\item $S_i$ is left reversible for all $i \in I$, that is, $S=S_c$ so that $S/\sim$ is a singleton.
	\item There exists a unique $i \in I$ such that $S_i$ is not left reversible, $\alpha_i\colon (S_i)_c \curvearrowright S_i/_\sim$ is almost free, and $S_j$ is left reversible for all $j \in I\setminus\{i\}$.
	\end{enumerate}
\end{enumerate}
\end{proposition}

In view of the direct sum decomposition for graph products over the coconnected components, we need to understand the behaviour of the graph product in the case of a coconnected graph with at least two vertices. To do this, we will need to consider a variant of the action $\alpha$ for $S^*$, i.e.~$\alpha^*\colon S^*\curvearrowright S/S^*, x.[s] := [xs]$. Also, we will assume that all vertex semigroups $S_v, v \in V$ are nontrivial in order to avoid pathological cases. For instance, if $\Gamma$ is the union of a complete graph and an isolated vertex $v$, and $S_v$ is trivial, then the graph product will be the direct sum of the right LCM semigroups attached to the vertices of the complete graph, even though the original graph was larger and coconnected.

\begin{thm}\label{thm:reg properties for coconnected graph products}
If $\Gamma=(V,E)$ is coconnected, $\lvert V\rvert \geq 2$, and $(S_v)_{v \in V}$ is a family of nontrivial right LCM semigroups, then the following assertions hold:
\begin{enumerate}[(i)]
\item $S_\Gamma^*$ is the graph product of $(S_v^*)_{v \in V}$, $(S_\Gamma^{\phantom{*}})_c=S_\Gamma^*$, and $(S_\Gamma^{\phantom{*}})_{ci} = S_\Gamma^{\phantom{*}} \setminus S_\Gamma^*$.
\item For $s,t \in S_\Gamma$, $s \sim t$ is equivalent to $s \in tS_\Gamma^*$.
\item $S_\Gamma$ is core factorable and $(S_\Gamma)_{ci}\subset S_\Gamma$ is $\cap$-closed.
\item The action $\alpha\colon S_\Gamma^*\curvearrowright S_\Gamma/_\sim$ is faithful if and only if $S_\Gamma$ is not a group.
\item The action $\alpha$ is almost free if and only if 
	\begin{enumerate}[(a)]
	\item $\alpha_v^*\colon S_v^*\curvearrowright S_v^{\phantom{*}}/S_v^*$ is almost free for every isolated vertex $v \in V$, and
	\item for every connected component $U\subset V$ with $\lvert U\rvert\geq 2$, either $S_u$ is a group for all $u \in U$ or $S_u^*$ is trivial for all $u \in U$.
	\end{enumerate}
\end{enumerate}
\end{thm}
\begin{proof}
For (i), let $s_{v(1)}s_{v(2)}\cdots s_{v(n)}$ be a reduced expression for $s\in S_\Gamma$. Clearly, $s$ is invertible in $S_\Gamma$ if and only if $s_{v(k)} \in S_{v(k)}^*$ for all $k$. The homomorphism from the graph product of $(S_v^*)_{v \in V}$ to $S_\Gamma$ (resulting from the universal property) is bijective, so that $S_\Gamma^*$ is the graph product with respect to $\Gamma$ and $(S_v^*)_{v \in V}$. Now assume that there is $1 \leq m \leq n$ such that $s_{v(k)} \in S_{v(k)}^*$ for $1 \leq k < m$ but $s_{v(m)} \notin S_{v(m)}^*$. Since $\Gamma$ is coconnected, there is $w \in V$ with $w \neq v(m)$ and $(v(m),w) \notin E$. For every $t \in S_w\setminus\{1\}$, we thus have $s_{v(m)}s_{v(m+1)}\cdots s_{v(n)} \perp ts_{v(m)}s_{v(m+1)}\cdots s_{v(n)}$. By left cancellation, this yields $s \perp s_{v(1)}s_{v(2)}\cdots s_{v(m-1)}ts_{v(m)}s_{v(m+1)}\cdots s_{v(n)}$, so that $s \notin (S_\Gamma)_c$. This proves $(S_\Gamma)_c=S_\Gamma^*$, and the claims $(S_\Gamma^{\phantom{*}})_{ci} = S_\Gamma^{\phantom{*}} \setminus S_\Gamma^*$, (ii), and (iii) are immediate consequences of this.

For (iv), we note that $\alpha$ is not faithful if $S_v$ is a group for all $v \in V$ because then $S_\Gamma/_\sim$ is a singleton while $S_\Gamma^*=S_\Gamma$ is nontrivial. So let us assume that there exists $v \in V$ with $S_v^{\phantom{*}} \neq S_v^*$. Every $x \in S_\Gamma^*\setminus\{1\}$ has a reduced expression $x_{u(1)}x_{u(2)}\cdots x_{u(m)}$ with $x_{u(k)} \in S_{u(k)}^*\setminus\{1\}$. Since $\Gamma$ is coconnected and $\lvert V\rvert \geq 2$, there exists a blocking path $w(1),\ldots,w(n)$ for $\{u(m)\}$ with $w(n)=v$, see Lemma~\ref{lem:blocking paths in coconnected graphs}. Choose $s_{w(k)} \in S_{w(k)}\setminus\{1\}$ for $1 \leq k <n$ and $s_{w(n)} \in S_{w(n)}^{\phantom{*}}\setminus S_{w(n)}^*$. Then $s:= s_{w(1)}s_{w(2)}\cdots s_{w(n)} \in S_\Gamma$ satisfies $x_{u(m)}s \perp s$. If $1 \leq k\leq m-1$ satisfies $(u(k),u(\ell)) \in E$ for all $k < \ell \leq m$, then $(u(k),u(m)) \in E$ in particular implies $u(k) \neq v(1)$. For the same reason, $(u(k),v(1)) \in E$ implies $u(k) \neq v(2)$, and so on. Thus $x_{u(1)}x_{u(2)}\cdots x_{u(m)} s_{w(1)}s_{w(2)}\cdots s_{w(n)}$ is a reduced expression for $xs$ and we conclude that orthogonality is not destroyed by $x_{u(1)}x_{u(2)}\cdots x_{u(m-1)}$, i.e.~$xs \perp s$. In particular, $[xs] \neq [s]$ and therefore $\alpha$ is faithful.

To prove (v), we first observe that (a) is necessary for $\alpha$ to be almost free: If $v \in V$ is isolated, then $[xs]=[s]$ for $x \in S_v^*\setminus\{1\}$ and $[s] \in S_\Gamma$ implies $s \in S_v$. Suppose next that (b) does not hold, i.e.~there exists a connected component $U\subset V$ of $\Gamma$ with $\lvert U\rvert\geq 2$ such that there are $u,v \in U$ with $S_v^{\phantom{*}}\neq S_v^*$ and $S_u^* \neq \{1\}$. If $u=v$, then we can pick $w \in U\setminus\{v\}$ with $(v,w) \in E$. If there is $x \in S_{w}^*\neq \{1\}$, then $[xs]=[sx]=[s]$ for all $s \in S_v$, and since $S_v^{\phantom{*}}/S_v^*$ is infinite, $\alpha$ fails to be almost free for $x$. On the other hand, $S_w$ is nontrivial, so $S_w^*=\{1\}$ implies that $S_w^{\phantom{*}}/S_w^*$ is infinite, and then almost freeness fails for every $x \in S_{v}^*\neq \{1\}$. 

Now suppose $u \neq v$. As $U$ is connected, we can find a path $v(0):=u,v(1),\ldots,v(n):=v$ from $u$ to $v$ inside $U$, i.e.~$(v(k),v(k+1)) \in E$ for all $0\leq k < n$. Then there exists $0 \leq k <n$ such that $S_{v(k)}^* \neq \{1\}$ and $S_{v(k+1)}^{\phantom{*}}\neq S_{v(k+1)}^*$, and we can apply the above argument to deduce that $\alpha$ is not almost free. We have thus proven that almost freeness of $\alpha$ implies (a) and (b).

Conversely, assume that (a) and (b) hold. If $S_\Gamma$ is a group, then there is nothing to show, so we may suppose that $S_\Gamma^{\phantom{*}} \neq S_\Gamma^*$. Let $x \in S_\Gamma^*\setminus\{1\}$ be presented by a reduced expression $x_{u(1)}x_{u(2)}\cdots x_{u(m)}$ with $x_{u(k)} \in S_{u(k)}^*\setminus\{1\}$. Fix $s \in S_\Gamma^{\phantom{*}} \setminus S_\Gamma^*$ with reduced expression $s_{v(1)}\cdots s_{v(n)}$, $s_{v(k)} \in S_{v(k)}$. Let $1 \leq j \leq n$ be the smallest number such that $s_{v(j)} \notin S_{v(j)}^*$. By (b), $j$ is invariant under shuffling and we know that $v(j)$ does not belong to the connected component of any $u(k)$ that emits an edge. Therefore, $xs \perp s$ and then $[xs] \neq [s]$, unless $j=m=1$ and $u(1)=v(1)=v$ for some isolated vertex $v \in V$. In this case, (a) says that there are only finitely many fixed points for $x$ in $S_v^{\phantom{*}}/S_v^*$. Thus $\alpha$ is almost free if (and only if) (a) and (b) hold.
\end{proof}

\begin{remark}\label{rem:finite propagation}
The graph product $S_\Gamma$ for a coconnected graph $\Gamma$ with $\lvert V\rvert \geq 2$ has finite propagation if $S_v^*$ is a finite group for all $v \in V$.
\end{remark}

Let us now summarise what Proposition~\ref{prop:reg properties for direct sums} and Theorem~\ref{thm:reg properties for coconnected graph products} imply for graph products of right LCM semigroups.

\begin{thm}\label{thm:reg prop for graph products}
Let $\Gamma=(V,E)$ be a graph and $(S_v)_{v \in V}$ a family of nontrivial right LCM semigroups. Then:
\begin{enumerate}[(i)]
\item $S_\Gamma^* = \bigoplus_{v \in V_u} S_v^* \oplus \bigoplus_{i \in I_2} S_{\Gamma_i}^*$.\vspace*{1mm}
\item $(S_\Gamma)_c = \bigoplus_{v\in V_u} (S_v)_c \oplus \bigoplus_{i \in I_2}S_{\Gamma_i}^*$.\vspace*{1mm}
\item $(S_\Gamma)_{ci}'= \bigoplus_{v\in V_u} (S_v)_{ci}' \oplus \bigoplus_{i \in I_2} S_{\Gamma_i}$.
\item Two elements $s,t \in S_\Gamma$ are core related if and only if $s_v \sim_v t_v$ for all $v \in V_u$ and $s_i \in t_iS_{\Gamma_i}^*$ for all $i \in I_2$.
\item $S_\Gamma$ is core factorable if and only if $S_v$ is core factorable for every $v \in V_u$.
\item $(S_\Gamma)_{ci} \subset S_\Gamma$ is $\cap$-closed if and only if $(S_v)_{ci} \subset S_v$ is $\cap$-closed for every $v \in V_u$.
\item The action $\alpha\colon (S_\Gamma)_c \curvearrowright S_\Gamma/_\sim$ is faithful if and only if $\alpha_v\colon (S_v)_c \curvearrowright S_v/_\sim$ is faithful for every $v \in V_u$, and for every $i \in I_2$ there exists $w \in V_i$ such that $S_w$ is not a group.
\item The action $\alpha\colon (S_\Gamma)_c \curvearrowright S_\Gamma/_\sim$ is almost free if and only if one of the following conditions holds: 
	\begin{enumerate}[(a)]
	\item $(S_v)_c = \{1\}$ for all $v \in V_u$ and $S_w^* = \{1\}$ for all $w \in V\setminus V_u$.
	\item $(S_v)_c \neq \{1\}$ for a unique $v \in V_u$ with $\alpha_v\colon (S_v)_c \to S_v/_\sim$ almost free, while $S_w=(S_w)_c$ for all $w \in V_u\setminus\{v\}$ and $S_{w'}^{\phantom{*}}=S_{w'}^*$ for all $w' \in V\setminus V_u$.
	\item $S_{\Gamma_i}^* \neq \{1\}$ for a unique $i \in I_2$ with $\alpha_i\colon S_{\Gamma_i}^* \to S_{\Gamma_i}^{\phantom{*}}/_\sim$ almost free, while $S_w=(S_w)_c$ for all $w \in V_u$ and $S_{\Gamma_j}^{\phantom{*}}=S_{\Gamma_j}^*$ for all $j \in I_2\setminus \{i\}$.
	\end{enumerate}
\item $S_\Gamma$ has finite propagation if $S_v$ has finite propagation for every $v \in V_u$ and $S_w^*$ is a finite group for all $w \in V\setminus V_u$. 
\end{enumerate}
\end{thm}

The conditions for almost freeness in Theorem~\ref{thm:reg prop for graph products} correspond to $(S_\Gamma)_c=\{1\}$, $(S_\Gamma)_c= (S_v)_c$, and $S_\Gamma^{\phantom{*}}/_\sim \cong = S_{\Gamma_i}/_\sim$, respectively. Hence they are quite restrictive, and we view this as an indication that finite propagation might be much more useful for graph products than almost freeness of $\alpha$, see \cite{ABLS}*{Theorem~4.2(2)} for details.

When applied to right-angled Artin monoids, Theorem~\ref{thm:reg prop for graph products} takes a simpler form:

\begin{corollary}\label{cor:reg prop for ra Artin monoids}
For a graph $\Gamma=(V,E)$, the right-angled Artin monid $A_\Gamma^+$ satisfies:
\begin{enumerate}[(i)]
\item $(A_\Gamma^+)^* = \{1\}, (A_\Gamma^+)_c = \bigoplus_{v \in V_u}\N$, and $(A_\Gamma^+)_{ci}^1= \bigoplus_{i \in I_2} A_{\Gamma_i}^+$.
\item Two elements $s,t \in A_\Gamma^+$ are core related if and only if $s_i = t_i$ for all $i \in I_2$.
\item $A_\Gamma^+$ is core factorable, $(A_\Gamma^+)_{ci} \subset A_\Gamma^+$ is $\cap$-closed, and $A_\Gamma^+$ has finite propagation.
\item The action $\alpha\colon (A_\Gamma^+)_c \curvearrowright A_\Gamma^+/_\sim$ is faithful if and only if $\Gamma$ has no universal vertex.
\item The action $\alpha\colon (A_\Gamma^+)_c \curvearrowright A_\Gamma^+/_\sim$ is almost free if and only if $V_u \in \{\emptyset,V\}$, i.e.~the core of $A_\Gamma^+$ is trivial or $A_\Gamma^+$ is the free abelian monoid in $V$.
\end{enumerate}
\end{corollary}

\section{The absence of property (AR)}\label{sec:absence}
In this section, we will prove that for many graph products of right LCM semigroups $S_\Gamma$, the only accurate foundation sets are given by elements of $S_\Gamma^*$. In particular, we obtain the an abundance of right LCM semigroups that lack property (AR). Again, the starting point is a basic observation for direct sums of right LCM semigroups, which allows us to boil the analysis down to the coconnected case:

\begin{proposition}\label{prop:AR for direct sums}
Let $(S_i)_{i \in I}$ be a family of right LCM semigroups. If $\bigoplus_{i \in I} S_i$ has property (AR), then $S_i$ has property (AR) for all $i \in I$.
\end{proposition}
\begin{proof}
Fix $i \in I$ and let $S:=\bigoplus_{i \in I} S_i$. Every foundation set $F$ for $S_i$ is a foundation set for $S$. Suppose that $F$ has an accurate refinement $F_a$ in $S$. For $s \in S_\Gamma$, we let $s=s_i + \hat{s}_i$ with $s_i \in S_i$ and $\hat{s}_i \in \bigoplus_{j \in I\setminus\{i\}} S_j$. If $s \in F_a$, then $\{ f_i \in S_i \mid f \in F_a: \hat{f}_i \not\perp \hat{s}_i\}$ is an accurate refinement for $F$ inside $S_i$.
\end{proof}

\begin{corollary}\label{cor:AR reduction for graph products}
If a graph product $S_\Gamma$ has property (AR), then $S_{\Gamma_i}$ has property (AR) for each coconnected component $\Gamma_i$ of $\Gamma$.
\end{corollary}

\begin{thm}\label{thm:lack of acc FS}
Let $\Gamma = (V,E)$ be a coconnected graph with at least two vertices and suppose $(S_v)_{v \in V}$ is a family of nontrivial right LCM semigroups. 
\begin{enumerate}[(i)]
\item If $\Gamma$ is infinite, then every foundation set for $S_\Gamma$ contains an invertible element. In particular, $S_\Gamma$ has property (AR) and $C^*(S_\Gamma) = \CQ(S_\Gamma)$.
\item If $\Gamma$ is finite and $E \neq \emptyset$, then the accurate foundation sets for $S_\Gamma$ correspond to $S_\Gamma^*$. In particular, $S_\Gamma$ has property (AR) if and only if $S_\Gamma$ does not admit a foundation set without invertible elements.
\end{enumerate}
\end{thm}
\begin{proof}
Both (i) and (ii) hold for trivial reasons if $S_\Gamma$ is a group, so we can assume that there exists $w \in V$ with $S_w^{\phantom{*}} \neq S_w^*$. Let $F \subset S_\Gamma$ be a finite subset without invertible elements. For every $f \in F$, we choose a reduced expression $f=f_{v(1)}\cdots f_{v(m_f)}$ with $m_f \in \N^\times$ and $f_{v(k)} \in S_{v(k)}$. 

Suppose first that $\Gamma$ is infinite. As $f \in S_\Gamma^{\phantom{*}}\setminus S_\Gamma^*$, there is a least $1 \leq k_f \leq m_f$ such that $f_{v(k_f)} \notin S_{v(k_f)}^*$. Then $C := \{ v \in V \mid f_{v(k)} \in S_v \text{ for some } f \in F, 1 \leq k \leq k_f\}$. Then $C$ is a finite set of vertices so that Lemma~\ref{lem:blocking paths in coconnected graphs} grants us a blocking path $w(1),\ldots,w(n)$ for $C$ ending in $w$. If we choose any $s_k \in S_{w(k)}^{\phantom{*}}\setminus \{1\}$ for $1 \leq k < n$ and $s_n \in S_w^{\phantom{*}}\setminus S_w^*$, then $s_1\cdots s_n \perp f$ for all $f \in F$ as $s_1\cdots s_n \perp f_{v(1)}\cdots f_{v(k_f)}$ by construction, see Remark~\ref{rem:blocking paths}. Therefore $F$ is not a foundation set. We conclude that every foundation set for $S_\Gamma$ contains an invertible element $x$, which clearly gives an accurate refinement $\{x\}$. So $S_\Gamma$ has property (AR), but if the only accurate foundation sets come from invertible elements, then the boundary relation $\sum_{f \in F} e_{fS_\Gamma} = 1$ becomes trivial so that $C^*(S) = \CQ(S)$.

Now let $\Gamma$ be finite, $E \neq \emptyset$, and assume $F$ to be accurate as well. We need to show that $F$ is not a foundation set. Without loss of generality, we can require that $f_{v(m_f)}$ is not invertible for all $f \in F$ because invertible ends do not play a role when it comes to intersections of right ideals. Since $F$ does not contain any invertibles, we have $\ell(f) \geq 1$ for all $f \in F$. Let $L:= \max_{f \in F} \ell(F)$, and choose $f \in F$ with $\ell(f)=L$. Then we have $f=st_v$ for some $v \in V, t_v \in S_v\setminus\{1\}$, and $s \in S_\Gamma$ with $\ell(s) =L-1$. We will first show that $v$ is isolated, and then use this together with $E\neq \emptyset$ to conclude that $F$ cannot be a foundation set. 

If $(v,u) \in E$ for some $u \in V$, we employ Lemma~\ref{lem:blocking paths in coconnected graphs} to obtain a blocking path $w(1),\ldots,w(n)$ for $C:= \{u\}\cup N_u$, and set $w(0):=u$. Next, choose $b_k \in S_{w(k)}^{\phantom{*}}\setminus S_{w(k)}^*$ for each $1\leq k \leq n$, and let $r \in S_u\setminus\{1\}$. It then follows that $srb \perp f$ for $b:=b_1\cdots b_n$. Moreover, we have $\ell(srb) \geq m+1$. This could be assumed by extending the path $w(0),\ldots,w(n)$ in $\Gamma^{\text{opp}}$, but actually holds true in any case. It then follows that whenever $f' \in F$ satisfies $f' \not\perp srb$, we have $srb \in f'S_\Gamma$. If $sr \in f'\Gamma$, then $f' \not\perp f \neq f'$ so that $F$ would not be accurate. The blocking path then forces $f' = srb_1\cdots b_k$ for some $1 \leq k \leq n$. However, we then get $f'\perp sr'b$ for every $r' \in S_u\setminus\{r\}$. Since $S_u$ is a left cancellative semigroup that is not a group, it is infinite. Thus there is $r \in S_u\setminus\{1\}$ such that $srb \perp f'$ for all $f' \in F$. 

We deduce from this that $F$ cannot be a foundation set if there exists $f \in F$ with $\ell(f)=L$ that does not end in a part from an isolated vertex. In particular, if $\Gamma$ does not have any isolated vertices, no accurate finite subset $F$ without invertible elements is a foundation set. Now suppose $\Gamma$ has an isolated vertex $\tilde{v}$, and let 
\[F' := \{f \in F \mid f_{v(k)} \in S_v \text{ for some } k \Rightarrow v \text{ is not isolated.}\},\] 
that is, the subset of $F$ consisting of those elements whose reduced expressions do not contain any part coming from an isolated vertex. As $E \neq \emptyset$ and the vertex semigroups are all nontrivial, the finite accurate set $F'$ is also non-empty. 

Suppose first that there is $\tilde{f} \in F'$ with $\tilde{f} \in S_v^{\phantom{*}}\setminus S_v^*$ for some $v\in V$. Since $F'$ is accurate and $(v,u) \in E$ for some $u \in V$, we have $s \notin f'S_\Gamma$ for all $s \in S_u$ and $f' \in F'$. Thus we get $str \perp f'$ for all $f' \in F'$ whenever $s \in S_u, t \in S_{\tilde{v}}$, and $r \in S_w^{\phantom{*}}\setminus S_w^*$, compare Remark~\ref{rem:blocking paths}. For $f \in F\setminus F'$, we have $strtr \perp f$ unless $f \in strtS_\Gamma$ because $\tilde{v}$ is isolated and $r$ is not invertible. Since $F$ is finite while $S_w^{\phantom{*}}\setminus S_w^*$ is infinite, we conclude that there are $s \in S_u, t \in S_{\tilde{v}}$, and $r \in S_w^{\phantom{*}}\setminus S_w^*$ such that $strtr \perp f$ for all $f \in F$. So $F$ is not a foundation set. 

On the other hand, if we have $\ell(f) \geq 2$ for every $f \in F'$, we pick a vertex $v$ that emits an edge. Then $s \notin fS_\Gamma$ for all $s \in S_v,f \in F'$, and thus $str \perp f$ for all $f\in F'$ whenever $s \in S_v, t \in S_{\tilde{v}}$, and $r \in S_w^{\phantom{*}}\setminus S_w^*$. As in the previous case, there are $s,t,r$ such that $strtr \perp f$ for all $f \in F$, and thus $F$ is not a foundation set.

Finally, if $F$ is a foundation set for $S_\Gamma$ with $F \cap S_\Gamma^* = \emptyset$, then every refinement $F'$ of $F$ satisfies $F' \cap S_\Gamma^* = \emptyset$ as well, and thus can never be accurate. On the other hand, every foundation set $F$ with $x \in F \cap S_\Gamma^*$ has an accurate refinement $\{x\}$.
\end{proof}

We point out that the assumptions in Theorem~\ref{thm:lack of acc FS} are modest means to avoid the somewhat pathological cases: $S_\Gamma=S_v$, the free product $S_\Gamma = *_{v \in V} S_v$, and the graph product of groups.

\begin{remark}\label{rem:FS in inf coconnected graphs}
By Theorem~\ref{thm:lack of acc FS}~(i), foundation sets of $S_\Gamma$ are governed by parts from the finite coconnected components in the following sense: Let $F$ be a foundation set for $S_\Gamma$ such that no propert subset of $F$ is a foundation set. If $s=s_{v(1)}\cdots s_{v(n)} \in F$ with $s_{v(k)} \in S_{v(k)}$, then $s_{v(k)} \notin S_{v(k)}^*$ implies that $v(k) \in V_i$ for some finite coconnected component $\Gamma_i=(V_i,E_i)$ of $\Gamma$.
\end{remark}

\begin{corollary}\label{cor:AR for graph products}
Let $\Gamma$ be a graph and $(S_v)_{v \in V}$ a family of nontrivial right LCM semigroups. If there is $i \in I_2$ for which $\Gamma_i=(V_i,E_i)$ is finite with $E_i \neq \emptyset$, $S_v$ is not a group for some $v \in V_i$, and there exists a foundation set $F$ for $S_{\Gamma_i}$ without invertible elements, then $S_\Gamma$ does not have property (AR).
\end{corollary}
\begin{proof}
The claim follows from combining Theorem~\ref{thm:lack of acc FS} with Corollary~\ref{cor:AR reduction for graph products}.
\end{proof}

The previous results apply nicely to right-angled Artin monoids.

\begin{corollary}\label{cor:AR for ra Artin monoids}
For graph $\Gamma$, the following conditions are equivalent: 
\begin{enumerate}[(1)]
\item Every finite coconnected component $\Gamma_i$ of $\Gamma$ is edge-free.
\item Every finitely generated direct summand of $A^+_\Gamma$ is free.
\item The right-angled Artin monoid $A^+_\Gamma$ has property (AR).
\end{enumerate}  
\end{corollary}
\begin{proof}
The equivalence of $(1)$ and $(2)$ is clear from the direct sum description of $A_\Gamma^+$ in Subsection~\ref{subsec:graph products}. From Remark~\ref{rem:FS in inf coconnected graphs} we infer that it suffices to obtain accurate refinements of foundation sets $F$ for $A_\Gamma^+$ with $F \subset \bigoplus_{v \in V_u} S_v \oplus \bigoplus_{i \in I_2: \lvert V_i \rvert < \infty} A_{\Gamma_i}^+$. But if $(2)$ holds, then the latter is just a direct sum of finitely generated free monoids, and clearly admits accurate refinements. So $(2)$ implies $(3)$. Finally, if $(3)$ is valid and $\Gamma_i=(V_i,E_i)$ is a coconnected component of $\Gamma$ with $2 \leq \lvert V_i\rvert <\infty$, then $\{a_v \mid v \in V_i\}$ is a foundation set for $A_{\Gamma_i}^+$ without invertible elements, so Corollary~\ref{cor:AR for graph products} forces $E_i = \emptyset$, that is, $(1)$ holds.
\end{proof}

By Corollary~\ref{cor:AR for ra Artin monoids}, there exist countably many mutually non-isomorphic, finitely generated right LCM semigroups without property (AR). As a final part of this section, we address the existence of a generalised scale for right-angled Artin monoids associated to finite graphs. The existence of a generalised scale turned out to be relevant for a standardised approach to study KMS-states on the semigroup $C^*$-algebra $C^*(A_\Gamma^+)$, see \cite{ABLS}. We first note that free monoids have a generalised scale only if they are finitely generated and nonabelian, in which case it is unique: 

\begin{proposition}\label{prop:unique gen scale on free monoids}
The free monoid $\IF_m^+$ in $2 \leq m < \infty$ generators admits a unique generalised scale $N\colon \IF_m^+ \to \N^\times$ given by $N(w)= m^{\ell(w)}$, where $\ell$ denotes the word length of $w \in \IF_m^+$.
\end{proposition}
\begin{proof}
The map $N$ is a generalised scale. On the other hand, let $\tilde{N}$ be a generalised scale on $\IF_m^+ = \langle a_1,\ldots,a_m\rangle$, and fix $1 \leq i \leq m$. Then $\tilde{N}(a_i)>1$ as $a_i$ is not part of $(\IF_m^+)_c = \{1\}$. By definition of $\tilde{N}$ and since $\sim$ is trivial, the set $\tilde{N}^{-1}(\tilde{N}(a_i))$ is an accurate foundation set for $\IF_m^+$ of cardinality $\tilde{N}(a_i)$ that contains $a_i$. If there was $1\leq j\leq m$, $j \neq i$ such that $\tilde{N}(a_j) \neq \tilde{N}(a_i)$, then the foundation set property would give a $w \in \tilde{N}^{-1}(\tilde{N}(a_i))$ such that $w \in a_ja_i\IF_m^+$. As this forces $\tilde{N}(w) \geq \tilde{N}(a_i)\tilde{N}(a_j) > \tilde{N}(a_i)=\tilde{N}(w)$, we arrive at a contradiction. Thus $\tilde{N}(a_j)=\tilde{N}(a_i)$ for all $j \neq i$. But as $\{a_1,\ldots,a_m\}$ is an accurate foundation set for $\IF_m^+$, we conclude that $\tilde{N}(a_i) =m$ for all $1 \leq i \leq m$.
\end{proof}

We call $m \in \bigoplus_{i \in I} \{k \in \N \mid 2 \leq k<\infty\}$ \emph{rationally independent} if for all distinct $k,k' \in \bigoplus_{i \in I} \N$, the supernatural numbers $\prod_{i \in I} m_i^{k_i}$ and $\prod_{i \in I} m_i^{k'_i}$ are distinct.

\begin{proposition}\label{prop:gen scale on direct sum of free monoids}
Let $M$ be a free abelian monoid, $m \in \bigoplus_{i \in I} \{k \in \N \mid 2 \leq k<\infty\}$ for some nonempty set $I$. Then $S:= M \oplus \bigoplus_{i \in I} \IF_{m_i}^+$ admits a generalised scale $N\colon S \to \N^\times$ if and only if $m$ is rationally independent. In this case, $N$ restricts to the unique generalised scales $N_i$ on $\IF_{m_i}^+$, and is therefore unique.
\end{proposition}
\begin{proof}
As $M=S_c=\ker N$ for any generalised scale $N$ on $S$, see \cite{ABLS}*{Proposition~3.6(i)}, we can focus on $(\IF_{m_i}^+)_{i \in I}$. Recall that $\IF_{m_i}^+$ is the free monoid in $m_i$ generators, which we denote by $a_{i,1},\ldots,a_{i,m_i}$. The strategy is to prove that 
\begin{enumerate}[(a)]
\item any generalised scale $N$ on $S$ restricts to $N_i$ on $\IF_{m_i}^+$, and
\item the homomorphism $N\colon S \to \N^\times$ arising from $(N_i)_{i \in I}$ is a generalised scale if and only if $m$ is rationally independent.
\end{enumerate}
For (a), suppose $S$ admits a generalised scale $N$ and fix $i \in I, 1 \leq k \leq m_i$. Then $N(a_{i,k})>1$ and there are $w_1,\ldots,w_{N(a_{i,k})-1} \in S$ such that $\{a_{i,k},w_1,\ldots,w_{N(a_{i,k})-1}\}$ is an accurate foundation set for $S$ contained in $N^{-1}(N(a_{i,k}))$. Let us decompose $w_\ell$ as $w_\ell = \hat{w}_\ell \oplus \check{w}_\ell \in \IF_{m_i}^+ \oplus \bigl(M_n \oplus \bigoplus_{j \in I\setminus\{i\}} \IF_{m_j}^+\bigr)$. Then $\{a_{i,k},\hat{w}_1,\ldots,\hat{w}_{N(a_{i,k})-1}\}$ is a foundation set for $\IF_{m_i}^+$ with $a_{i,k} \perp \hat{w}_\ell$ and $N(\hat{w}_\ell)\leq N(a_{i,k})$ for all $\ell$. This forces $\{a_{i,k},\hat{w}_1,\ldots,\hat{w}_{N(a_{i,k})-1}\} \supset \{a_{i,1},\ldots,a_{i,m_i}\}$, and thus $N(a_{i,\ell}) \leq N(a_{i,k})$ for all $1 \leq \ell \leq m_i$, just like in the proof of Proposition~\ref{prop:unique gen scale on free monoids}. As $k$ was arbitrary, we deduce $N(a_{i,k})=m_i=N_i(a_{i,k})$ for all $i,k$.

In view of (a), the question behing the main claim becomes: Under which condition is the homomorphism $N\colon S \to \N^\times$ arising from the family of generalised scales $(N_i)_{i \in I}$ itself a generalised scale? If $m$ is rationally independent, then every $k \in N(S)$ has a factorization $k = \prod_{i \in I} m_i^{k_i}$ with uniquely determined $k_i \in \N$. This implies that 
\[\begin{array}{c} N^{-1}(k) = \{ t \oplus \bigoplus\limits_{i \in I} w_i \mid t \in M, w_i \in \IF_{m_i}^+ \text{ with } \ell_i(w_i) = k_i\}. \end{array}\]
Therefore, $\lvert N^{-1}(k)/_\sim \rvert = k$, and any transversal of $N^{-1}(k)/_\sim$ is an accurate foundation set for $S$, that is, $N$ is a generalised scale. On the other hand, if there are $k,k' \in \bigoplus_{i \in I}\N, k \neq k'$ such that $K:= \prod_{i \in I} m_i^{k_i} = \prod_{i \in I} m_i^{k'_i}$, then both $k$ and $k'$ yield a set of $K$ mutually orthogonal elements $s_1,\ldots,s_K \in S$ and $t_1,\ldots,t_K \in S$, respectively, with $N(s_j)=K=N(t_j)$ for all $j$. Since there is $i \in I$ with $k_i\neq k_i'$, the $i$-th components of $s_j$ and $t_{j'}$ have different length for all $j,j'$. Thus $s_j \not\sim t_{j'}$ for all $j,j'$, and we get $\lvert N^{-1}(K)/_\sim \rvert \geq 2K$. Therefore $N$ is not a generalised scale in this case.
\end{proof}

We can now state our conclusions for right-angled Artin monoids.

\begin{corollary}\label{cor:gen scale for ra Artin monoids}
For every graph $\Gamma$, the right-angled Artin monoid $A^+_\Gamma$ admits a generalised scale $N$ if and only if $V_u \neq V$, all coconnected components $\Gamma_i=(V_i,E_i)$ are finite and edge-free, and $\bigoplus_{i \in I_2} \lvert V_i\rvert$ is rationally independent. In this case, $N$ is unique.
\end{corollary}
\begin{proof}
The condition $V_u \neq V$ is equivalent to saying that $A_\Gamma^+$ is non-abelian, i.e.$I_2 \neq \emptyset$. So if all coconnected components $\Gamma_i=(V_i,E_i)$ are finite and edge-free, then $A_\Gamma^+ \cong \bigoplus_{v \in V_u} \N \oplus \bigoplus_{i \in I_2} \IF_{\lvert V_i\rvert}^+$. Hence, Proposition~\ref{prop:gen scale on direct sum of free monoids} implies that $A_\Gamma^+$ has a (unique) generalised scale $N$ if and only if $\bigoplus_{i \in I_2} \lvert V_i\rvert$ is rationally independent.

Conversely, suppose $A_\Gamma^+$ admits a generalised scale $N$. Since $N$ is a nontrivial homomorphism with $\ker~N = \bigoplus_{v \in V_u} \N$, we need to have $V_u \neq V$ so that the set $I_2$ is non-empty. Moreover, $A_\Gamma^+$ has property (AR) by \cite{ABLS}*{Proposition~3.6}, so Corollary~\ref{cor:AR for ra Artin monoids} implies that all finite coconnected components $\Gamma_i$ of $\Gamma$ are edge-free. If there was an infinite coconnected component $\Gamma_i=(V_i,E_i)$, then $1 <N(a_v) < \infty$ for all $v \in V_i$, and the defining property of a generalised scale would yield an accurate foundation set of the form $\{a_v,f_1,\ldots,f_{N(a_v)-1}\}$ for suitable $f_k \in A_\Gamma^+$. However, this contradicts Remark~\ref{rem:FS in inf coconnected graphs}, and we conclude that $\Gamma_i$ is finite for all $i \in I_2$. But then $A_\Gamma^+$ is covered by Proposition~\ref{prop:gen scale on direct sum of free monoids}, and it follows that $\bigoplus_{i \in I_2} \lvert V_i\rvert$ is rationally independent.
\end{proof}

\begin{corollary}\label{cor:admissible ra Artin monoids}
For every graph $\Gamma$, the right-angled Artin monoid $A^+_\Gamma$ is admissible if and only if it admits a generalised scale.
\end{corollary}
\begin{proof}
According to Corollary~\ref{cor:reg prop for ra Artin monoids}~(iii), $A_\Gamma^+$ is core factorable and $(A_\Gamma^+)_{ci}\subset A_\Gamma^+$ is $\cap$-closed, no matter what $\Gamma$ is. By Corollary~\ref{cor:gen scale for ra Artin monoids}, the conditions characterising the existence of a generalised scale $N$ include rational independence of $\bigoplus_{i \in I_2} \lvert V_i\rvert$. This feature implies $\text{Irr}(N(A_\Gamma^+)) = \{ \lvert V_i\rvert \mid i \in I_2\}$ and that this set freely generates $N(A_\Gamma^+)$, which is the last extra condition for admissibility.
\end{proof}

\end{document}